\documentclass[12pt]{article}
\usepackage{amsgen, amsmath, amsfonts, amsthm, amssymb, enumerate,amscd,hyperref,graphicx}
\usepackage[dvipsnames]{xcolor}
\hypersetup{colorlinks=true,linkcolor=blue,citecolor=magenta}
\usepackage[all]{xy}

\newtheorem{defn}{Definition}[section]
\newtheorem{prop}[defn]{Proposition}

\newtheorem{lem}[defn]{Lemma}
\newtheorem{thm}[defn]{Theorem}
\newtheorem{conjecture}[defn]{Conjecture}
\newtheorem*{conj}{Conjecture}
\newtheorem*{conj*}{Conjecture}

\newcommand {\C}{{\mathbb C}}

\newcommand {\HH}{{\mathfrak  H}}

\title{}
\author{ \\
}

\begin{document}

\title{Eichler cohomology and zeros of polynomials associated to derivatives of $L$-functions}
\author{Nikolaos Diamantis (University of Nottingham) 
\\
Larry Rolen (Hamilton Mathematics Institute \& Trinity College Dublin)
}

\maketitle
\begin{abstract}
In recent years, a number of papers have been devoted to the study of roots of period polynomials of modular forms. Here, we 
study cohomological analogues of the Eichler-Shimura period polynomials corresponding to higher $L$-derivatives.
We state general conjectures about the locations of the roots of the full and odd parts of the polynomials, in analogy with the 
existing literature on period polynomials, and we also give numerical evidence that similar results hold for our higher derivative ``period polynomials'' in the case of cusp forms.
We prove a special case of this conjecture in the case of Eisenstein series.

\end{abstract}

\section{Introduction}
Derivatives of $L$-functions, and especially higher order derivatives, remain mysterious objects. This is despite
intense work on key conjectures about them, such as those of
Beilinson, Birch--Swinnerton-Dyer, etc. Beilinson's conjecture, as formulated in \cite{KZ},
relates values of derivatives of $L$-functions to fundamental objects called periods. In \cite{KZ}, 
a complex number is called a {\it period} if its real and imaginary parts have the form
$$\int_V \frac{P(\mathbf{x})}{Q(\mathbf{x})}d\mathbf{x},$$
where $V$ is a domain in $\mathbb{R}^n$ defined by polynomial inequalities with coefficients in $\mathbb{Q}$
and $P, Q \in \mathbb{Q}[X_1, \dots, X_n].$ This definition accounts for numbers that are clearly very important
for number theory and other areas of mathematics, but are not necessarily algebraic, e.g.
$$\pi=\iint\limits_{x^2+y^2 \le 1} dx dy \qquad \text{and} \, \, \log(n)=\int_1^n \frac{1}{x}dx \, \, (\text{for} \, \,  
n \in \mathbb{N}).$$
Denote the set of periods by $\mathcal{P}$. 
A special case of Beilinson's conjecture in a version given in \cite{KZ} can be
stated as follows. 

\noindent
\begin{conj}[Delinge-Beilinson-Scholl]
Let $f$ be a weight $k$ Hecke eigencuspform for SL$_2(\mathbb{Z})$, $L_f(s)$ its $L$-function, and $m$ an integer. Then, if $r$ is the order of vanishing of $L_f(s)$ at $s=m$,  we have
$$L^{(r)}(m) \in \mathcal P [1/\pi ].$$
\end{conj}

The cases of $r=0$ and $m=1, \dots, k-1$ (that is, the case of {\it critical values} of $L_f(s)$) have been treated by Manin, Deligne
and others (e.g. \cite{M, De}). 
The case when $r=0$ and $m> k-1$ has been proven by Beilinson and Deninger-Scholl (see \cite{KZ} and the references therein) 
but for $r>0$ the picture is much less clear. Fundamental results by Gross-Zagier \cite{GZ}
and others in the context of the
Birch--Swinnerton-Dyer conjecture give insight for $r=1$ and $k=2$, but very little is known for $r \ge 2$.

One of the important tools for studying critical values is the {\it period polynomial}
$$\sum_{n=0}^{k-2} \binom{k-2}{n} i^{1-n}\Lambda_f(n+1)z^{k-2-n},$$
where
$\Lambda_f(s):=(2 \pi)^{-s}\Gamma(s) L_f(s)$ denotes the completed $L$-function (e.g. \cite{M, KoZ, Z}). Background for the
period polynomial will be discussed in the 
next section.

In this paper, we offer the following conjecture  about an analogue of this polynomial for all derivatives of $L_f(s)$ and then prove it in the case of Eisenstein series.
\begin{conjecture}\label{Conj0} (``Riemann hypothesis for period polynomials attached to derivatives of $L$-functions")
For any Hecke eigenform of weight $k$ on $\operatorname{SL}_2(\mathbb Z)$, and for each $m\in\mathbb Z_{\geq0}$, the polynomial 
\[
Q_f(z):=\sum_{n=0}^{k-2} \binom{k-2}{n} i^{1-n}\Lambda_f^{(m)}(n+1)z^{k-2-n}
\]
has all its zeros on the unit circle. 
Moreover, the ``odd part'' 
\[\sum_{\substack{n=1 \\ n \, \, \text{odd}}}
^{k-3}\binom{k-2}{n} i^{1-n}\Lambda_f^{(m)}(n+1)z^{k-3-n}
\]
has all of its zeros on the unit circle, except for a trivial zero at $z=0$ and $4$ zeros at the points $\pm a,\pm1/a$ for some real number $a$.
\end{conjecture}
The reason for calling it a ``Riemann Hypothesis" is that if $z_0$ is a zero of $Q_f$, then $-1/z_0$ is also a zero, as implied by the transformation
$$Q_f(-1/z)=Q_f(z)z^{2-k}.$$
Thus, the unit circle is the natural line of symmetry in this case (and can be mapped to a ``zeta polynomial'' where the Riemann Hypotheses stipulates that the roots lie on the line $\operatorname{Re}(s)=\frac12$ in a natural framework due to Manin \cite{Ma} 
and expounded upon by the authors of \cite{ORS}). 
This transformation law is an implication (see Lemma \ref{mthderiv} and \eqref{Q_f}) 
of the cohomological structure we associate to derivatives of $L$-functions and which formed the conceptual basis for the conjecture.  

{\bf Example}
We conclude with a numerical example of the conjecture, which is discussed further in Section~\ref{Gen}. Consider the  normalized weight $20$, level $1$ cuspidal eigenform $f$. A rescaling and change of variables in the second polynomial in Conjecture~\ref{Conj0} yields a polynomial approximately given by 
\[
z^{17}-5.805z^{15}+9.685z^{13}-6.720z^{11}+6.720^7-9.685z^5+5.805z^3-z.
\]
The roots predicted by Conjecture~\ref{Conj0} in this case are at $z=0$, at points on the unit circle with arguments approximately $0$, $13.5$ and $43$ degrees, together with the real and complex conjugates of these points, and seems to have roots at $\pm a^{\pm 1}$ for $a$ approximately $1.9$.

\section{Motivation, background, and structure of the paper}
As mentioned above, the values of $L_f(s)$ themselves inside the critical strip are better
understood than those of the derivatives. Important tools that have been used in their study are the
period polynomial and the Eichler cohomology.

Specifically, if $f$ is a cusp form of weight $k$ for $\Gamma:=\text{SL}_2(\mathbb{Z})$, it is possible to associate a $1$-cocycle to it as follows. 
We consider the action $|_{2-k}$ of $\Gamma$ on the space $\mathcal O$ of holomorphic 
maps on the upper half plane $\HH$, defined, for each $f\colon \HH \to \C$, by 
\begin{equation}\label{action}
(f|_{2-k} \gamma)(\tau):=f(\gamma \tau) j(\gamma, \tau)^{k-2}, \qquad \tau \in \HH, \gamma \in \Gamma,
\end{equation}
where 
$$j\left ( \left ( \begin{smallmatrix} * & * \\ c & d \end{smallmatrix} \right ), \tau \right ):=c\tau+d.$$
We use the same notation for the restriction of this action to the space $P_{k-2}$ of polynomial 
functions of degree $\le k-2$.

Then, for some $\tau_0 \in \mathfrak{H} \cup \mathbb{Q}\cup \{\infty\},$ 
we assign to our cusp form $f$ a map $\sigma_f$ that sends $\gamma \in \Gamma$ to 
$$\int_{\tau_0}^{\gamma^{-1}(\tau_0)}f(\tau)(\tau-z)^{k-2} d\tau.$$
The map $\sigma_f$ is then a $1$-cocycle and its cohomology class is independent of the choice of $\tau_0$.
We call this an {\it Eichler cocycle}.

Then, taking $\tau_0=0$, the {\it period polynomial} of $f$ can be recovered as the value of $\sigma_f$ at $S:=
\left ( \begin{smallmatrix} 0 & -1 \\ 1 & 0 \end{smallmatrix} \right )$:
$$\sigma_f(S)=\int_0^{\infty} f(\tau) (\tau-z)^{k-2}d\tau.$$
Indeed, as will be shown in greater generality in the sequel, this equals
$$\sum_{n=0}^{k-2} \binom{k-2}{n} i^{1-n}\Lambda_f(n+1)z^{k-2-n}.$$

The relations satisfied by the period polynomial as a value of an Eichler cocycle (including those originating in
the Hecke action compatibility of Eichler cocycles), have far-reaching consequences for the arithmetic and
geometry of $f$. For instance, a fairly immediate implication of the Eichler-Shimura isomorphism applied to the
period polynomial is Manin's Periods Theorem \cite{Ma}, which provides important information about 
the arithmetic nature of critical $L$-values.

The fundamental nature of the period polynomial is reflected in other aspects of its structure, 
for instance when viewed as a polynomial. In particular, the location of its roots has been studied by various 
authors. It seems that the first case to be considered was the analogue of the period
polynomial associated to Eisenstein series. Even to formulate the correct definition of period polynomials for 
Eisenstein series has been an important question of independent interest. The versions that will play a role in 
this paper are those of D. Zagier \cite{Z} and of F. Brown \cite{Br}.

With the definition as in \cite{Z}, M. R. Murty, C.J. Smyth and R.J.Wang \cite{MSW} 
proved that all non-real zeros of the odd part of the period polynomial of 
an Eisenstein series lie on the unit circle. This is a natural line of symmetry for the period polynomials, given that they are reciprocal polynomials thanks to the functional equation for completed $L$-values, and so as explained by S. Jin, W. Ma, K. Ono, and K. Soundararajan in \cite{JMOS}, such results can be thought of as a sort of ``Riemann Hypothesis'' for period polynomials. The interested reader is also referred to \cite{ORS} for further results connecting such results to Manin's theory of ``zeta polynomials'' $Z_f(s)$, which are transformed versions of the period polynomials sending the unit circle to the line $\operatorname{Re}(s)=\frac12$ and satisfying the functional equation $Z_f(1-s)=\pm Z_f(s)$. M.N. Lal\'\i n and C.J. Smyth \cite{LS} proved that all roots of the
full period polynomial of an Eisenstein series are unimodular. Analogous results have been proved
for the case of cusp forms. For example, J.B. Conrey, D.W. Farmer and \"O. Imamoglu \cite{CFI} 
have proved that, apart from some ``trivial" real roots, all roots of the odd part of the 
period polynomial of a cusp form lie on the unit circle. (In the context of Conjecture~\ref{Conj0} above, these are exactly the points stated there with $a=2$). 
A. El-Guindy and W. Raji \cite{ElG-R} have extended this to the full period polynomial by
showing that its roots are all unimodular. The analogues of these results for higher levels, together 
with very explicit approximations for the exact locations of the roots is proved in \cite{JMOS}. 
 
An unexpected observation that we first made numerically was that the unimodularity of the roots 
also occurs for certain ``period polynomials" attached to derivatives of $L$-functions defined by D. Goldfeld and 
the first author. Motivated by the success of Eichler cohomology and the (classical) period polynomials, 
they defined analogues of the period polynomial that encode values of {\it derivatives} of $L$-functions \cite{G, D1, D}.

Specifically, with the same notation as above, we noted that, for the many examples of weights
we numerically tested, the polynomial 
$$-\sum_{n=0}^{k-2} \binom{k-2}{n} i^{1-n}\Lambda_f'(n+1)z^{k-2-n}$$
has its roots on the unit circle. Likewise, the ``odd part" of this polynomial seems to have 
roots all on the unit circle other that 5 simple ones of a shape resembling the main results of \cite{CFI}.

That was surprising because, as pointed out in \cite{KZ}, only the first non-vanishing 
derivatives are normally expected to have some number-theoretic significance. However, 
in the polynomial in question, all integer values of
derivatives inside the critical strip play an equal role in our results.

To develop a general framework for these phenomena, we had to conceptually justify our choice of the ``period polynomial" encoding 
values of derivatives of $L$-functions. That was achieved by cohomological 
considerations on the basis of the classical Eichler theory and the constructions of \cite{D}.

We first (see Section \ref{ppolyE}) reinterpret the cocycles associated to values of $L$-functions of a general modular 
form $f$ in a way that will be consistent with our corresponding construction for derivatives of $L$-functions. Our cocycle is
``canonical" in the sense that it belongs to the same cohomology class
as the image of $f$ under the Eichler-Shimura isomorphism (see Proposition \ref{canon}). Furthermore, it 
turns out that, in our interpretation, the period polynomial we associate to Eisenstein series coincides with
the version of \cite{Br}, rather than that of \cite{Z}. 

The form of the cocycles we assign to derivatives of general modular forms (Sec. \ref{derppolyE}) 
is entirely analogous to that of the cocycles we attach to values of $L$-functions. As is normally to 
be expected in the case of (higher) derivatives, one has to separate the ``main" from the ``lower order" terms. 
A second feature indicating that our cocycles are the ``right'' objects to look at is that the 
passage from first to second and higher derivatives is achieved via a group cohomological
construction based on cup products, thus allowing for a unified treatment of all derivatives. 

With these ``period polynomials for derivatives of $L$-functions" in place, we turn our attention to their zeros. We
prove special cases of our conjecture (which is stated in Setion. \ref{Gen}) for Eisenstein series. We only treat the analogues
of \cite{MSW} and \cite{CFI}, by focusing on the ``odd parts" of these polynomials. 

The additional numerical evidence for the truth of the
Conjecture~\ref{Conj0} in the cuspidal case, as well as further discussion, is given in Section \ref{Gen}.

\section*{Acknowledgements}
The authors are grateful to Kathrin Bringmann, Francis Brown, Dorian Goldfeld, Ken Ono, and Jesse Thorner for useful discussions, as well as to Trinity College Dublin 
for hosting the first author on a visit where this project arose.

\section{Cocycles associated to values of $L$-functions and to values of their derivatives}\label{ppolyE}
The classical Eichler-Shimura theory assigns a cocycle to a modular form $f$ in such a way that
it characterizes the critical values of $L_f(s)$. In the case of the modular group this
cocycle is determined by its value at the involution which is called ``period polynomial." This
was originally defined and studied for cusp forms yielding many important arithmetic results.

Zagier \cite{Z} seems to be the first one to study an extension of the period polynomial to 
non-cuspidal forms. His definition has been generalized and modified by various authors in 
accordance with different perspectives. We will recall one of them in the next subsection.

Another direction in which Eichler-Shimura theory and period polynomials have been extended 
is to values of derivatives of $L$-functions of cusp forms. Goldfeld \cite{G} and the first author \cite{D1, D, D2}
have considered
an approach allowing for the encoding values of derivatives into analogues of the period polynomial 
and which can be interpreted in the context of Eichler cohomology. 

In this section, we extend the constructions of \cite{D} and \cite{G} to non-cuspidal modular 
forms. Before doing that,  
we consider the cocycles associated to values of $L$-functions of general modular 
forms in a formulation that fits the ``period polynomials" we will associate 
to derivatives of $L$-functions in the next subsection. 

\subsection{Cocycles associated to values of $L$-functions}
In this subsection, we reformulate the known theory of cocycles associated to values of $L$-functions 
of general modular forms so that it is consistent with the corresponding construction for derivatives of $L$-functions
that we will discuss in the next subsection.

Throughout, let $k$ be an even positive integer. 
We will be using the action $|_{2-k}$ of $\Gamma$ on $\mathcal O$ defined by \eqref{action}
and its restriction to the space $P_{k-2}$ of polynomial 
functions of degree $\le k-2$.

As usual, we denote the space of $i$-cochains for $\Gamma$ with coefficients in a right $\Gamma$-module $M$
by $C^i(\Gamma, M).$ We will also use the formalism of ``bar resolution" for the differential $d^i\colon C^i(\Gamma, M) \to 
C^{i+1}(\Gamma, M)$:
\begin{align}\label{differential}
\begin{split}
&(d^i \sigma)(g_1, \dots, g_{i+1}):=\\
&
\sigma(g_2,\dots,g_{i+1}).g_1 
+\sum_{j=1}^i (-1)^j \sigma(g_1,\dots, g_{j+1}g_j, \dots, g_{i+1})
+(-1)^{i+1}\sigma(g_1,\dots,g_{i}).
\end{split}
\end{align}

We now start the construction of the cocycles we will assign to a modular form.
Let $$f(\tau)=\sum_{n=0}^{\infty} a_n e^{2 \pi i n \tau}$$
be an element of the space $M_k$ of modular forms of weight $k$ for $\Gamma$.
As usual, we define its $L$-function, for Re$(s) \gg 1$, by
$$L_f(s):=\sum_{n=1}^{\infty}\frac{a_n}{n^s}$$
and, the ``completed" $L$-function by
$$\Lambda_f(s):=(2 \pi)^{-s} \Gamma(s) L_f(s).$$
It is well-known (see, e.g., \cite{I}, Chapt. 7) that 
$\Lambda_f$ has a meromorphic continuation to the entire complex plane with possible
(simple) poles at $0$ and $k$ and that it
 satisfies the functional equation
\begin{equation}\label{FE}
\Lambda_f(s)= i^k \Lambda_f(k-s).
\end{equation}
It further has an integral expression:
\begin{equation}\label{Mellin} \Lambda_f(s)=
\int_1^{\infty}(f(iv)-a(0))v^{s-1} dv+
i^k \int_1^{\infty}(f(iv)-a(0))v^{k-s-1} dv-
\frac{a(0)}{s}-\frac{a(0)i^k}{k-s}
\end{equation}
which, in the cuspidal case, reduces to the classical 
expression for $\Lambda_f$ as a Mellin transform.

Define $v_f \in C^0(\Gamma, \mathcal{O})$ by
$$v_f(z):=\int_{\infty}^z (f(w)-a_0)(w-z)^{k-2}dw+\frac{a_0}{k-1} z^{k-1}.$$
This is well-defined because of the exponential decay of $f(w)-a_0$ at $\infty$.
Set $\sigma_f:=d^0v_f.$
The next lemma shows that this is a $1$-cocycle in Eichler cohomology.
\begin{lem} The map $\sigma_f$ takes values in $P_{k-2}.$ In particular, it
gives a $1$-cocycle in $P_{k-2}$.
\end{lem}
\begin{proof}
We first note that
$z^{k-1}/(k-1)=\int_0^z (w-z)^{k-2}dw.$ Therefore, the value of $\sigma_f(\gamma)=v_f|_{2-k}(\gamma-1)$ at $z$ equals
\begin{align}\label{cocycle0}
\begin{split}
& \left(\int_i^z f(w)(w-z)^{k-2}dw\right)\Big|_{2-k}(\gamma-1)-a_0 \left( \int_i^z (w-z)^{k-2}dw \right)\Big|_{2-k}(\gamma-1)\\
&+
 \int_{\infty}^i (f(w)-a_0)(w-z)^{k-2}dw\Big|_{2-k}(\gamma-1)+a_0 \left( \int_0^z (w-z)^{k-2}dw \right)\Big|_{2-k}(\gamma-1).
 \end{split}
\end{align}
The third term is clearly in $P_{k-2}$. On the first, we use the elementary identity
\begin{equation}\label{(w-z)}
(w-\gamma z) j(\gamma, z)=(\gamma^{-1} w-z)j(\gamma^{-1}, w)
\end{equation}
and the modular invariance of $f$ to get
\begin{align}\label{cocycle}
\begin{split}
&\int_i^{\gamma z} f(w)(w-\gamma z)^{k-2}j(\gamma, z)dw-
\int_i^{ z} f(w)(w- z)^{k-2}dw\\
&=
\int_{i}^{\gamma z} f(\gamma^{-1} w)(\gamma^{-1}w-z)^{k-2}d(\gamma^{-1} w)-
\int_i^{ z} f(w)(w- z)^{k-2}dw=
\int_{\gamma^{-1} i}^{i} f(w)(w- z)^{k-2}dw,
\end{split}
\end{align}
which is in $P_{k-2}.$ The remaining two terms of \eqref{cocycle0} combine to $a_0(\int_0^i (w-z)^{k-2}dw)|_{\gamma-1}$, which is also in $P_{k-2}.$

Further, since $\sigma_f$ is given as the differential of a $1$-cochain, it will satisfy 
the $1$-cocycle relation. 
\end{proof}

This cocycle 
belongs to the cohomology class associated to $f$ under the Eichler Shimura isomorphism
$$\phi: \overline S_k \oplus M_k \xrightarrow{\sim} H^1(\Gamma, P_{k-2}).$$
This isomorphism is induced by the assignment of $f \in M_k$ to the map $\phi(f)$ such that 
$$\phi(f)(\gamma)=\int_{\gamma^{-1} i}^i f(w)(w-z)^{k-2}dw \qquad \text{for $\gamma \in \Gamma$}$$
Furthermore, the value of the cocycle $\sigma_f$ encodes the critical $L$-values. Specifically, we have the following result.
\begin{prop}\label{canon}
If $f \in M_k$, then the following are true.
\begin{enumerate}
\item[i).] The $1$-cocycle $\sigma_f$ is a representative of the cohomology class
of $\phi(f)$.
\item[ii).] For each $z \in \C$, we have
\begin{equation*}
\sigma_f(S)=-i \sum_{j=0}^{k-2} \binom{k-2}{j} (iz)^j \Lambda_f(j+1).
\end{equation*}
\end{enumerate}
\end{prop}
\begin{proof}
We begin with the proof of i). Using \eqref{cocycle0} and \eqref{cocycle}, we see that 
\begin{align}\label{Eichlercocycle}
\begin{split}
\sigma_f(\gamma)
&=\int_{\gamma^{-1} i}^{i} f(w)(w- z)^{k-2}dw\\
&+
\left(\int_{\infty}^i (f(w)-a_0)(w-z)^{k-2}dw
+a_0 \int_0^i (w-z)^{k-2}dw \right)\Bigg|_{2-k}(\gamma-1)
.
\end{split}
\end{align}
Since the part inside the parentheses is in $P_{k-2}$, the second row of \eqref{Eichlercocycle} is a coboundary and thus
$\sigma_f$ differs from $\phi(f)$ by a coboundary.

We now turn to the proof of ii). Setting $\gamma=S$ in \eqref{Eichlercocycle}, we see that 
$\sigma_f(S)$ equals 
$$\int_{\infty}^i (f(w)-a_0)\left ( (wz+1)^{k-2}-(w-z)^{k-2} \right )dw
+\frac{a_0}{k-1}\left((i-z)^{k-1}-(-z)^{k-1} \right)\Big|_{2-k}(S-1).$$
The binomial expansion and \eqref{Mellin} imply the identity after an elementary calculation. 
\end{proof}

\bf Remark \rm When $f=E_k$, the value of $\sigma_f$ at $S$ is essentially the period polynomial $\textrm{p} (E_k)$ 
which plays an important role in \cite{Br}. As pointed out in
\cite{Br}, this differs from the ``extended period polynomial" of \cite{Z} in that we remain 
in $P_{k-2}$, whereas in \cite{Z}, the representation has to be extended to a larger space 
that includes non-polynomial functions. This is achieved by the addition of the term 
$\frac{a_0}{k-1} z^{k-1}$ in the definition of $v_f$, which at first may seem unnatural.

From this viewpoint, \eqref{Eichlercocycle} can be thought of as a formula for 
the Eichler cocycle whose specialization at $S$ is  Brown's period polynomial of 
an Eisenstein series.  

\subsection{Cocycles associated to values of derivatives of $L$-functions}\label{derppolyE}
We maintain the notation of the previous section. 
We further set 
$u(\tau):=\log(\eta(\tau))$
where $$\eta(\tau):=e^{\frac{2 \pi i \tau}{24}} \prod_{n=0}^{\infty} (1- e^{2 \pi i n \tau})$$
is the Dedekind eta function. For each $\gamma \in \Gamma$, this function satisfies
\begin{equation}\label{v}
u(\gamma \tau)=u(\tau)+\log (j(\gamma, \tau))+c_{\gamma}
\end{equation}
for some $c_{\gamma} \in \C$. In particular, $c_S=-\frac{\pi i}{2}$.

Define the cochain $v_f \in C^1(\Gamma, \mathcal{O})$ by
\begin{align*}
v_f(\gamma)
&
:=\int_{\infty}^z (f(w)-a_0)(w-z)^{k-2} 
\left ( u(\gamma w)-u(w) \right )dw
\\
&+
a_0 \int_i^z (w-z)^{k-2}\left ( u(\gamma w)-u(w) \right )dw.
\end{align*}
Set $\sigma_f:=d^1v_f.$
In this case, $\sigma_f=d^1v_f $ is a $2$-cocycle in Eichler cohomology.
\begin{lem}\label{Lemfirst}The map $\sigma_f$ 
takes values in $P_{k-2}$ and thus gives a $2$-cocycle in $P_{k-2}$.
\end{lem}
\begin{proof}
The value of $v_f(\gamma_2)$ at $z$ equals
$$\int_i^{z} f(w) (w-z)^{k-2} \left ( u(\gamma_2 w)-u(w) \right )dw+
\int_{\infty}^i (f(w)-a_0) (w-z)^{k-2} \left ( u(\gamma_2 w)-u(w) \right )dw.$$
The image of the second term under the differential $d^1$ (see \eqref{differential}) is clearly in $P_{k-2}.$
The image of the first term equals
\begin{align*}
\begin{split}
& \int_i^{\gamma_1 z} f(w)(w-\gamma_1 z)^{k-2}j(\gamma_1, z)^{k-2} \left ( u(\gamma_2 w)-u(w) \right )dw \\
& -\int_i^{z} f(w)(w-z)^{k-2} \left ( u(\gamma_2 \gamma_1 w)-u(\gamma_1 w) \right )dw.
\end{split}
\end{align*}
With \eqref{(w-z)} and the modularity of $f$, we deduce that this equals
$$-\int_i^{\gamma_1 i} f(w)(w-z)^{k-2} \left ( u(\gamma_2 \gamma_1 w)-u(\gamma_1 w) \right )dw,$$
which is in $P_{k-2}.$
\end{proof}
The next proposition shows that when $f$ is cuspidal, $\sigma_f$ coincides with the cocycle 
associated to derivatives of $L$-functions 
in \cite{D}.
\begin{prop} Let $f$ be a cusp form of weight $k$ for $\Gamma$.  Then 
\begin{align*}
\begin{split}
\sigma_f(\gamma_1, \gamma_2)
&=\int_{\gamma^{-1}_1 \infty}^{\infty} f(w)(w-z)^{k-2}(u(\gamma_2 w)-u(w))dw\\
&=\int_{\infty}^{\gamma_1 \infty} f(w)(w-z)^{k-2}(u(\gamma_2 w)-u(w))dw\Big |_{2-k}\gamma_1.
\end{split}
\end{align*}
\end{prop}
\begin{proof} When $f$ is cuspidal, then
$$v_f(\gamma)=\int_{\infty}^z f(w)(w-z)^{k-2} 
\left ( u(\gamma w)-u(w) \right )dw$$
and this implies that
\begin{align*}
\begin{split}
\sigma_f(\gamma_1, \gamma_2)
&=\int_{\infty}^{\gamma_1 z} f(w)(w-\gamma_1 z)^{k-2}j(\gamma_1, z)^{k-2} 
\left ( u(\gamma_2 w)-u(w) \right )dw \\
&-\int_{\infty}^z f(w)(w-z)^{k-2}
\left ( u(\gamma_2 \gamma_1 w)-u(\gamma_1 w) \right )dw.
\end{split}
\end{align*}
Equation \eqref{(w-z)} and the modularity of $f$ imply that, after the change of variables $\gamma^{-1} w \to w$, the first integral equals
$$\int_{\gamma_1^{-1}\infty}^{z} f(w)(w-z)^{k-2} 
\left ( u(\gamma_2 \gamma_1 w)-u(\gamma_1 w) \right )dw,$$
which gives the first equality. The second follows from a change of variables and \eqref{(w-z)}.
\end{proof}
 We will now show that, up to a simple multiple of a fixed polynomial, $\sigma_f$ encodes 
the values of derivatives of the $L$-function of $f$ inside the critical strip just as the cuspidal analogue in 
\cite{D} did.
\begin{prop}\label{1stder}Set $$P(z)=
\sum_{n=0}^{k-2} \binom{k-2}{n} \frac{i^{1-n}}{(n+1)^{2}}z^{k-2-n}.$$ Then
$$\sigma_f(S, S)=
-\sum_{n=0}^{k-2} \binom{k-2}{n} i^{1-n}\Lambda_f'(n+1)z^{k-2-n}
+a(0) (P|_{2-k}(1+S))(z).$$
\end{prop}
\begin{proof} Upon differentiation of \eqref{Mellin}, we obtain 
$$\Lambda'_f(s)=\int_1^{\infty}(f(iv)-a(0))v^{s-1}\log (v) dv-
i^k \int_1^{\infty}(f(iv)-a(0))v^{k-s-1}\log (v) dv+
\frac{a(0)}{s^2}-\frac{a(0)i^k}{(k-s)^2}.$$
Equation \eqref{v} gives $\log(v)=u(S iv)-u(iv)$, and so
\begin{align*} 
&
\sum_{n=0}^{k-2} \binom{k-2}{n} i^{1-n}\Lambda_f'(n+1)z^{k-2-n}
=\\
&
\sum_{n=0}^{k-2} \binom{k-2}{n} i^{1-n}z^{k-2-n} \left(
i^{-n-1}\int_i^{i \infty}(f(w)-a(0)) w^n (u(S w)-u(w))dw\right.\\
&
\left.
-i^{n+1}\int_i^{i \infty}(f(w)-a(0)) w^{k-n-2} (u(S w)-u(w))dw  
- a(0) \left ( \frac{i^k}{(k-n-1)^2}-\frac{1}{(n+1)^2} \right ) \right )\\
&=
\int_i^{i \infty}(f(w)-a(0)) (z-w)^{k-2} (u(S w)-u(w))dw
\\
&
+\int_i^{i \infty}(f(w)-a(0)) (zw+1)^{k-2} (u(S w)-u(w))dw \\
&
+a(0) (P(z)+P(-1/z)z^{k-2})=-\sigma_f(S, S)+a(0)(P|_{2-k}(1+S))(z).
\end{align*}
\end{proof}
\bf Remark. \rm It is possible to eliminate $P$ from the statement Proposition~\ref{1stder} 
by modifying the definition of $v_f$ to match, in some respects, even more
perfectly the $v_f$ we assigned to the values of $L$-functions in the
previous subsection. However, the formula would become more complicated without
an obvious benefit.

\subsection{Higher derivatives}
We can now extend the construction above to account for all derivatives of $L$-functions.
It will be convenient to use the group cohomology formalism of cup products. We consider the cup product map
$$\cup\colon C^1(\Gamma, \mathcal O) \otimes
C^m(\Gamma, \mathcal O) \to C^{m+1}(\Gamma, \mathcal O) $$
given by 
$$\left ( \phi_1 \cup \phi_2 \right )(\gamma_1, \gamma_2, \dots, \gamma_{m+1}):=\phi_1(\gamma_1) \left ( \phi(\gamma_2, \dots, \gamma_{m+1})|_0 \gamma_1\right ).$$
For $\phi_i \in C^1(\Gamma, \mathcal{O})$, we consider the iterated product:
$$\phi_1 \cup \dots \cup \phi_n:=\phi_1 \cup \left(\phi_2 \cup \left( \dots ( \phi_{n-1} \cup \phi_n ) \dots \right)\right) \in C^n(\Gamma, \mathcal O).$$
An important property is that that cup products of cocycles are cocycles.

With this notation, we set, for $n \in \mathbb N$, 
$$V_n:=v \cup v \cup \dots \cup v \qquad \text{($n$ times)}$$
where $v$ is $1$-cocycle given by
$\gamma \to u|_0(\gamma-1)$ (with $u$ as in the last subsection). 
As mentioned above, this will be a $n$-cocycle.

Let $v_f \in C^n(\Gamma, \mathcal{O})$ be given by
\begin{align*}
v_f(\gamma_1, \dots, \gamma_n)
&=\int_{\infty}^z (f(w)-a_0)(w-z)^{k-2} V_n(\gamma_1, \dots, \gamma_n)(w)dw\\
&+ a_0\int_{i}^z (w-z)^{k-2} V_n(\gamma_1, \dots, \gamma_n)(w)dw.
\end{align*}
Setting $\sigma_f:=d^n v_f$, we arrive at the following analogue of Lemma~\ref{Lemfirst} for higher cocycles.

\begin{lem}\label{mthderiv} The map $\sigma_f$ 
takes values in $P_{k-2}$ and thus gives an $(n+1)$-cocycle in $P_{k-2}$.
\end{lem}
\begin{proof}
The value of $v_f(\gamma_2, \dots, \gamma_{n+1})$ at $z$ equals
\begin{equation}
\label{decomp_n}
\int_i^z f(w) (w-z)^{k-2} V_n(\gamma_2, \dots, \gamma_{n+1})(w)dw+
\int_{\infty}^i (f(w)-a_0) (w-z)^{k-2} V_n(\gamma_2, \dots, \gamma_{n+1})(w)dw.
\end{equation}
The image of the second term under the differential $d^{n}$ is clearly in $P_{k-2}.$
The image of the first term equals
\begin{multline}\label{firstterm}
\int_i^{\gamma_1 z} f(w)(w-\gamma_1 z)^{k-2}j(\gamma_1, z)^{k-2} V_n(\gamma_2, \dots, \gamma_{n+1})(w)dw 
+\int_i^z f(w)(w-z)^{k-2} \times \\
\left \{ \sum_{j=1}^n (-1)^j V_n(\gamma_1,\dots, \gamma_{j+1}\gamma_j, \dots, \gamma_{n+1})(w)+(-1)^{n+1}V_n(\gamma_1,\dots,\gamma_{n})(w) \right \} dw.
\end{multline}
Since $V_n$ is a $n$-cocycle in terms of the action of $|_0$ on $\mathcal O$, the part inside the curly brackets equals
$$-V_n(\gamma_2, \dots, \gamma_{n+1})(\gamma_1 w).$$
On the other hand, \eqref{(w-z)} and the modularity of $f$ imply that the first integral in \eqref{decomp_n} equals
$$\int_{\gamma_1^{-1} i}^z f(w)(w-z)^{k-2} V_n(\gamma_2, \dots, \gamma_{n+1})(\gamma_1 w)dw.$$
Thus, \eqref{firstterm} equals
$$\int_{\gamma_1^{-1} i}^i f(w)(w-z)^{k-2} V_n(\gamma_2, \dots, \gamma_{n+1})(\gamma_1 w)dw,$$
which is in $P_{k-2}.$
\end{proof}
Finally, we describe the relation with the higher derivatives of $L$-functions in the critical strip.
\begin{prop}\label{higherderivative} For each $m \in \mathbb{N}$, set $$P(z)=
\sum_{n=0}^{k-2} \binom{k-2}{n} \frac{i^{1-n}}{(-n-1)^{m+1}}z^{k-2-n}.$$ Then
\begin{equation}\label{Q_f}
(-1)^{m}\sigma_f(S, \dots S)=
\sum_{n=0}^{k-2} \binom{k-2}{n} i^{1-n}\Lambda_f^{(m)}(n+1)z^{k-2-n}
-a(0) m! (P|_{2-k}(1+(-1)^{m+1}S))(z),
\end{equation}
where $\sigma_f$ has $m+1$ arguments. 
\end{prop}
\begin{proof} As before, we have
\begin{align*}
\begin{split}
\Lambda_f^{(m)}(s)&=\int_1^{\infty}(f(iy)-a(0))y^{s-1}\log^{m} (y) dy\\
&+(-1)^m i^k \int_1^{\infty}(f(iy)-a(0))y^{k-s-1}\log^{m} (y) dy+
\frac{a(0) m!}{(-s)^{m+1}}+\frac{a(0)m!i^k}{(s-k)^{m+1}}.
\end{split}
\end{align*}
We deduce the result mutatis mutandis by working as in Proposition \ref{1stder}.
\end{proof}

\section{Zeros of ``period polynomials"}\label{0s}
In this section we will prove that, in the special case that $f$ is an Eisenstein series,
the zeros of the ``odd part" of the ``period polynomials" we have attached to 
values of derivatives of $L$-functions lie on the unit circle. 

To highlight more clearly the key ideas, we study on its own the case of first derivative 
and then show how this can be generalized to higher derivatives.


\subsection{The case of the first derivative }
We will now prove the analogue of Theorem 5.1. of \cite{MSW} for first derivatives. For simplicity, 
we focus on the case of weight $k\equiv0\pmod 4$. 
Consider the piece of $\sigma_f(S, S)$ in Proposition \ref{1stder} that contains the first derivatives
of $L$-functions, i.e.,
\begin{equation}\label{derivpiece}
-\sum_{n=0}^{k-2} \binom{k-2}{n} i^{1-n}\Lambda_f'(n+1)z^{k-2-n}
\end{equation}
We will study the part of this polynomial that captures the values of $\Lambda_f'$ at even arguments
$$P_f(z):=
\sum_{\substack{n=1 \\ n \, \, \text{odd}}}
^{k-3} \binom{k-2}{n} i^{1-n}\Lambda_f'(n+1)z^{k-3-n}.$$
(For simplicity, we factor out a $-z$ from \eqref{derivpiece} because it only affects trivially the zeros
of the polynomial.)
\begin{thm}\label{1stderE} Let $k \equiv 0 \pmod 4$ and let
$$f(\tau):=E_k(\tau)=1+\frac{(2 \pi)^k}{\zeta(k)\Gamma(k)}\sum_{n=1}^{\infty} \sigma_{k-1}(n)e^{2 \pi i n \tau}.$$
Then all zeros of $P_f(z)$ lie on the unit circle.
\end{thm}
\begin{proof}
By \eqref{FE} , our polynomial can be expressed by
$$-\sum_{\substack{n=1 \\ n \, \, \text{odd}}}^{k-3} \binom{k-2}{n} i^{n-1}\Lambda_f'(n+1)z^{n-1}=
-\sum_{n=0}^{\frac{k}{2}-2} \binom{k-2}{2n+1} i^{2n}\Lambda_f'(2n+2)z^{2n}.$$
The assertion of the theorem is then equivalent to the statement that the zeros of
$$-\sum_{n=0}^{\frac{k}{2}-2} \binom{k-2}{2n+1} \Lambda_f'(2n+2)(-z)^n.$$
are all on the unit circle. 

Since, by \eqref{FE}, $\Lambda'_f(k/2)=0$, we have that
\begin{align*}
\begin{split}
&-\sum_{n=0}^{\frac{k}{2}-2} \binom{k-2}{2n+1} \Lambda_f'(2n+2)(-z)^n\\
=&-\sum_{n=0}^{\frac{k}{4}-2} \binom{k-2}{2n+1} \Lambda_f'(2n+2)(-z)^n
-\sum_{n=\frac{k}{4}}^{\frac{k}{2}-2} \binom{k-2}{2n+1} \Lambda_f'(2n+2)(-z)^n
 \\
&-\sum_{n=0}^{\frac{k}{4}-2} \binom{k-2}{2n+1} \Lambda_f'(2n+2)(-z)^n+
q_f(z)
=-q_f(1/z)z^{\frac{k}{2}-2}+q_f(z)
\end{split}
\end{align*}
where
$$q_f(z):= -\sum_{n=\frac{k}{4}}^{\frac{k}{2}-2} \binom{k-2}{2n+1} \Lambda_f'(2n+2)(-z)^n.
$$
(At the last step we made the change of variables $n \to \frac{k}{2}-2-n$ and used \eqref{FE} and that $k/2-2$ is even.)
The theory of ``self-inversive" polynomials (see, for instance Th. 2.2 of \cite{ElG-R})
implies that it suffices to show that all zeroes of $q_f(z)$
are in $|z| \le 1.$

To show this, we first re-write $\Lambda_f'(2n+2)$. It is well-known that the $L$-function of $E_k$ equals $\zeta(s)\zeta(s-k+1)$ and,
thus, together with the functional equation for $\zeta(s)$ (in the form with the sine function) we have
\begin{align}\label{FEz}
\begin{split}
\Lambda_f(s)&=\frac{\Gamma(s)}{(2 \pi)^s}\zeta(s)\frac{2^{s-k+1}}{\pi^{k-s}}\Gamma(k-s)\sin\left(\frac{\pi}{2}(s-k+1)\right)\zeta(k-s)
\\ 
&=2 (2 \pi)^{-k} \cos(\frac{\pi s}{2}) \Gamma(s) \Gamma(k-s)\zeta(s) \zeta(k-s).
\end{split}
\end{align}
Therefore, 
\begin{equation}\label{Lambda'}
\Lambda_f'(s)=\Lambda_f(s) \left(\frac{-\pi \sin(\frac{\pi s}{2})}{2\cos(\frac{\pi s}{2})}+
\frac{\Gamma'(s)}{\Gamma(s)}
-
\frac{\Gamma'(k-s)}{\Gamma(k-s)}+\frac{\zeta'(s)}{\zeta(s)}-\frac{\zeta'(k-s)}{\zeta(k-s)}\right),
\end{equation}
and thus for $n \ge 0$  we see that 
$\Lambda'_f(2n+2)$ equals 
\begin{align*} 
&\frac{2}{(2 \pi)^k} (2n+1)! (k-2n-3)! \cos(\pi(n+1)) \zeta(2n+2) \zeta(k-2n-2) \\ 
&\times\left(H_{2n+1}-H_{k-2n-3}+\frac{\zeta'(2n+2)}{\zeta(2n+2)}-
\frac{\zeta'(k-2n-2)}{\zeta(k-2n-2)}\right).
\end{align*}
Here $H_n:=\sum_{l=1}^{n}1/l$, and we used (5.4.14) of \cite{NIST}.

Thus, 
\begin{align*} q_f(z)&= \frac{-2 (k-2)!}{(2 \pi)^k} \sum_{n=\frac{k}{4}}^{\frac{k}{2}-2} \zeta(2n+2)\zeta(k-2n-2) (-1)^{n+1} \\ 
&\times\left(H_{2n+1}-H_{k-2n-3}+\frac{\zeta'(2n+2)}{\zeta(2n+2)}-
\frac{\zeta'(k-2n-2)}{\zeta(k-2n-2)}\right) (-z)^n .
\end{align*}
It is now clear that $H_n \le H_m$ when $n \le m$, and that $\zeta'(s)/\zeta(s)$ is negative and increasing for $s>1$ follows from the well-known Dirichlet series expansion
\begin{equation}\label{vonM}
\frac{\zeta'(s)}{\zeta(s)}=-\sum_{n=1}^{\infty} \frac{\Lambda(n)}{n^s},
\end{equation}
where $\Lambda(n)$ is the von Mangoldt function.
 Therefore, for $n, m$ with $\frac{k}{4} \le n \le m \le \frac{k}{2}-2$, 
we have 
\begin{align*}
\begin{split} 
0 & \le 
H_{2n+1}-H_{k-2n-3}+\frac{\zeta'(2n+2)}{\zeta(2n+2)}-
\frac{\zeta'(k-2n-2)}{\zeta(k-2n-2)}  
\\
&\le
H_{2m+1}-H_{k-2m-3}+\frac{\zeta'(2m+2)}{\zeta(2m+2)}-
\frac{\zeta'(k-2m-2)}{\zeta(k-2m-2)}
.
\end{split}
\end{align*}
It further follows that $f(x):=\zeta(2x+2)
\zeta(k-2x-2)$ is increasing in $[\frac{k}{4}-1, \frac{k}{2}-2]$ because then
\begin{equation}\label{a_n}
f'(x)=2f(x) \left(\frac{\zeta'(2x+2)}{\zeta(2x+2)}-
\frac{\zeta'(k-2x-2)}{\zeta(k-2x-2)}\right)>0.
\end{equation}
Therefore the coefficients of $z^n$ in $q_f(z)$ form a non-negative and increasing 
sequence and thus, the Enestr\"om-Kakeya Theorem \cite{En,Ka} 
implies that the zeros of $q_f(z)$
are all in $|z|\le 1$. Therefore,
Theorem 2.2 of \cite{ElG-R} implies the theorem.
\end{proof}

\subsection{Zeros of polynomials associated to higher derivatives}
For $j=0, 1, \dots$ set
\[
\Psi_{j+1}(s):=\frac{\partial^j}{\partial s^j}\left (
\frac{\Gamma'(s)}{\Gamma(s)}
-\frac{\Gamma'(k-s)}{\Gamma(k-s)} \right)
\]
and
\[Z_{j+1}(s):=
\frac{\partial^j}{\partial s^j}\left (\frac{\zeta'(s)}{\zeta(s)}
-\frac{\zeta'(k-s)}{\zeta(k-s)} \right ).
\]
We further
set $f(s):=-\frac{\pi}2\cdot\tan(\pi s/2)$ and we note that for even integers $s$, and for $j\in\mathbb N$,
\[
f^{(j)}(s)=\begin{cases}0& \text{if} j\in2\mathbb N, 
\\
\frac{(-1)^{\frac{j+1}2}B_{j+1}(2^{j+1}-1)\pi^{j+1}}{(j+1)}& \text{if} j\in2\mathbb N+1,
\end{cases}
\]
where $B_n$ denotes the $n$-th Bernoulli number.
Since it does not depend on $s \in 2 \mathbb{Z}$,
denote this constant by $b_{j+1}$ (so that $b_1=0$, $b_2=-\pi/2$, $b_3=0,\ldots$). 

An iterated application of the Leibniz rule to \eqref{Lambda'} implies, by induction, that 
\begin{equation}\label{finalLambda}
\Lambda_f^{(m+1)}(s)=\Lambda_f(s) \sum\nolimits^* c_{i_1, \dots, j_1, \dots, k_1, \dots, n_{i_1}, \dots, n_{j_1} \dots} b_{i_1}^{n_{i_1}} \dots \Psi_{j_1}^{n_{j_1}}(s) \cdots Z_{k_1}^{n_{k_1}}(s) \cdots
\end{equation}
for some  $c_{i_1, \dots, j_1, \dots, k_1, \dots, n_{i_1}, \dots, n_{j_1} \dots} \ge 0$,
where the star indicates that the sum ranges over all positive $i_1, \dots, j_1, \dots, k_1, \dots , n_{i_1}, \dots, n_{j_1}, \dots$ such that 
$$i_1 n_{i_1}+\dots+j_1 n_{j_1}+\dots+k_1 n_{k_1}+\dots=m+1.$$

We consider the part of the right hand side of \eqref{finalLambda} that does not include the lower order terms arising from the constants $b_j$:
\[\Lambda_f(s) \sum\nolimits^* c_{j_1, \dots, k_1, \dots, n_{i_1}, \dots, n_{j_1} \dots} \Psi_{j_1}^{n_{j_1}}(s) \dots Z_{k_1}^{n_{k_1}}(s) \dots.
\]
Working as above, \eqref{FEz} shows that the $(m+1)$-st derivative of 
$$\widetilde{\Lambda}_f(s):=\frac{\Lambda_f(s)}{\cos(\frac{\pi s}{2})}$$
equals that part, i.e.,
$$\widetilde{\Lambda}_f^{(m+1)}(s)=\widetilde{\Lambda}_f(s) \sum\nolimits^* c_{j_1, \dots, k_1, \dots, n_{i_1}, \dots, n_{j_1} \dots} \Psi_{j_1}^{n_{j_1}}(s) \dots Z_{k_1}^{n_{k_1}}(s) \dots.$$

In view of this, set $$P^m_f(z):=\sum_{\substack{n=1 \\ n \, \, \text{odd}}}
^{k-3} \binom{k-2}{n} i^{1-n}\widetilde{\Lambda}_f^{(m)}(n+1)z^{k-3-n}.$$
Then we have
\begin{thm} Let $k \equiv 0 \pmod 4$ and $m \ge 1.$ Then all zeros of $P^{m}_f(z)$ lie on the unit circle.
\end{thm}
\begin{proof}
Equation \eqref{FEz} implies that 
$$\widetilde{\Lambda}_f(s)=\widetilde{\Lambda}_f(k-s)$$ and therefore
$\widetilde{\Lambda}_f
^{(m)}(s)=(-1)^{m}\widetilde{\Lambda}_f^{(m)}(k-s).$
Thus the claim is equivalent to the unimodularity of the zeros of 
$$\sum_{n=0}^{\frac{k}{2}-2} \binom{k-2}{2n+1} \widetilde{\Lambda}_f^{(m)}(2n+2)z^n.$$
As in the proof of Th. \ref{1stderE}, this polynomial can be expressed as
$$q^m_f(z)+(-1)^m q^m_f(1/z)z^{\frac{k}{2}-2}$$
where
$$q^m_f(z):= \sum_{n=\frac{k}{4}-1}^{\frac{k}{2}-2}\delta_n \binom{k-2}{2n+1} \widetilde{\Lambda}_f^{(m)}(2n+2)z^n
$$
where $\delta_n=1/2$ if $n=k/4-1$, and $1$ otherwise.
With Th. 2.2 of \cite{ElG-R}, for the proof of unimodularity of the zeroes $P^{m}_f(z)$ it suffices
to show that the zeros of $q^m_f$ are in $|z| \le 1$.
 
To prove this we will use the following 
\begin{lem}\label{monot} For each $j=1, 2, \dots$, the functions $\Psi_j(s)$ and $Z_j(s)$ are positive and 
increasing in $[k/2, k-2]$.
\end{lem}
\begin{proof}
Eq. (5.15.1) of \cite{NIST} implies 
\[
\Psi_{j+1}(s)=j!\sum_{r\geq0}\left(\frac{(-1)^{j+1}}{(s+r)^{j+1}}+\frac{1}{(r+k-s)^{j+1}}\right).
\]
It is enough to show that each term in this series is positive: 
This is obvious for $s \in [k/2, k-2]$ when $j$ is odd. If $j$ is even, we have, for each $r\geq0$, 
\[s+r>r+k-s>0 \qquad \text{ and thus} \, \,
\frac{1}{(r+k-s)^{j+1}}-\frac1{(s+r)^{j+1}}>0,
\]
for $s \in (k/2, k-2].$ (The monotonicity at $k/2$ follows by continuity).

For the monotonity of $Z_j$, we note that \eqref{vonM} implies 
\[
Z_{j+1}(s)=\sum_{r\geq1}\Lambda(r)\log^j(r)\cdot\left(\frac{(-1)^{j+1}}{r^s}+\frac{1}{r^{k-s}}\right)>0
\]
for all $j$. The positivity of each term follows trivially for $j$ odd and from the inequality $r^s>r^{k-s}$ when 
$j$ is even
\end{proof}
Returning to the proof of the theorem, we see that 
\begin{multline}
q_f^m(z)= \\ \sum_{n=\frac{k}{4}-1}^{\frac{k}{2}-2} \binom{k-2}{2n+1} 
\delta_n \widetilde{\Lambda}_f(2n+2) \left ( \sum\nolimits^* c_{j_1, \dots, k_1, \dots, n_{i_1}, \dots, n_{j_1} \dots} \Psi_{j_1}^{n_{j_1}}(2n+2) \dots Z_{k_1}^{n_{k_1}}(2n+2) \dots. \right ) z^n \\
\frac{2 (k-2)!}{(2 \pi)^k}\sum_{n=\frac{k}{4}-1}^{\frac{k}{2}-2}\delta_n a_n
\left ( \sum\nolimits^* c_{j_1, \dots, k_1, \dots, n_{i_1}, \dots, n_{j_1} \dots} \Psi_{j_1}^{n_{j_1}}(2n+2) \dots Z_{k_1}^{n_{k_1}}(2n+2) \dots. \right ) z^n,
\end{multline}
where $a_n=\zeta(2n+2)\zeta(k-2n-2).$ As shown by \eqref{a_n}, $a_n$ and thus $\delta_n a_n$ is increasing as $n$ ranges from $\frac{k}{4}-1$ to $\frac{k}{2}-2$.
Also, the term inside the brackets is a linear combination, with positive coefficients, of products of $Z_j$ and $\Psi_j$ which, by Lemma \ref{monot}, are increasing in the range of interest. Therefore, the coefficients of $z^n$ form an increasing and positive sequence and thus,
by the Enestr\"om-Kakeya (see e.g. Th. 1.3 of \cite{GG} Theorem we deduce the result.
\end{proof}

\section{Conjecture~\ref{Conj0} in the cuspidal case}\label{Gen}
The results above for polynomials associated to Eisenstein series motivate the analogous statement for the ``period polynomial" associated  to general modular forms in Proposition \ref{higherderivative}. 
For convenience, we recall the precise formulation of this statement, which was given as Conjecture~\ref{Conj0} above.
\begin{conj*}
For any Hecke eigenform of weight $k$ and level $1$, and for each $m\in\mathbb Z_{\geq0}$, the polynomial
\[
\sum_{n=0}^{k-2} \binom{k-2}{n} i^{1-n}\Lambda_f^{(m)}(n+1)z^{k-2-n}
\]
has all its zeros on the unit circle. 
Moreover, the ``odd part'' 
\[\sum_{\substack{n=1 \\ n \, \, \text{odd}}}
^{k-3}\binom{k-2}{n} i^{1-n}\Lambda_f^{(m)}(n+1)z^{k-3-n}
\]
has all of its zeros on the unit circle, except for a trivial zero at $z=0$ and $4$ zeros at the points $\pm a,\pm1/a$ for some real number $a$.
\end{conj*}
\bf Remark. \rm 
In \cite{CFI}, the conjecture for the odd part is shown to be true when $m=0$ with $a=2$. There, it is shown that these ``trivial zeros'' arise in a natural way from the Eichler-Shimura relations. It would be interesting to see if similar Eichler-Shimura-type relations for higher derivative period polynomials explain the nature of these numbers $a$. It also seemed, numerically, that $a$ only depends on the weight and the number of $L$-derivatives taken, and not on the particular eigenform, which would suggest such an approach is plausible.
In addition to our proof of the first part of this conjecture in the special case of Eisenstein series, there is both theoretical and
numerical evidence for the truth of the full conjecture. 

As mentioned in the introduction, results on the zeros of (classical) period polynomials for 
cusp forms in \cite{CFI,ElG-R,JMOS} etc. extend analogous results about Eisenstein series in \cite{LS,MSW}. 
Furthermore, these results, taken together, cover both the ``odd" part and the full period polynomial, just as our
conjecture does. 

Numerically, we found using SAGE (and in particular Dokchitser's $L$-function calculator) that the conjecture holds for the full period polynomials up to a precision of $10^{-10}$ on the norms of the zeros of each of these polynomials for $m\leq3$ and for all eigenforms of level $1$ and weight $k\leq 50$, and a number of sufficiently large weight examples where tested for the odd polynomials to make the numerical evidence convincing. 

As commented above, this seems somewhat surprising, as usually the ``interesting'' information of modular $L$-functions is in their first non-vanishing central derivative, and no such restriction is made here in the study of these higher $L$-derivative values. Significant theoretical bounds seem to be required to prove this conjecture, even in the first unproven case of $m=1$, for example. In particular, proofs like those in \cite{JMOS} require non-negativity results for central $L$-values, which seem to be harder for higher derivatives. It would also be interesting to present, even in the known case of $m=0$, a uniform proof which covers all cases simultaneously without breaking into finitely many cases and checking the remaining ones numerically. This would be nice theoretically, in order to understand the ``reason'' why the previous results on period polynomials are true, and would be important in order to prove a general conjecture as stated here, where one would like to study infinitely many iterated derivatives (and so cannot numerically verify finitely many cases on each). 

To highlight the difficulty in applying results like the Enestr\"om-Kakeya Theorem above, and why the our conjecture may seem surprising, consider the example of the unique normalized weight $20$, level $1$ cuspidal eigenform $f$. The odd part of the period polynomial built out of first $L$-derivatives, after changing variables and rescaling to be monic, is approximately 
\[
z^{17}-5.805z^{15}+9.685z^{13}-6.720z^{11}+6.720^7-9.685z^5+5.805z^3-z
\]
Writing this polynomial in the form $q_{f}(1/z)z^{18}+q_{f}(z)$ as above in the proof of the analogous result for Eisenstein series, we might like to apply the Enestr\"om-Kakeya Theorem. However, we can see that the coefficients of $q^{o}_{f}$ in this case are oscillating instead of monotonic. 

In addition to proving and explaining the uniform nature of the conjecture above, it would also be interesting to describe the consequences for Eichler-Shimura cohomology. That is, what are the applications of our theorems above and of the conjecture in the case of cusp forms? 
Moreover, is there a suitable theory of Manin ``zeta-polynomials'' as discussed in \cite{ORS}?


\begin{thebibliography}{99}

 \bibitem{BCD}
R. Bruggeman, Y. Choie, and N. Diamantis, \emph{Holomorphic automorphic forms and cohomology}, Memoirs of the AMS (to appear)
arXiv:1404.6718 
 
 \bibitem{Br}
F. Brown, \emph{Multiple Modular Values and the relative completion of the fundamental group of $M_{1,1}$}, preprint,
 arXiv:1407.5167.
 
 \bibitem{CFI} J.B. Conrey, D.W. Farmer, and {\"O}. Imamo{\=g}lu, {\it The nontrivial zeros of period polynomials of modular forms lie on
the unit circle}, Int. Math. Res. Not. no. 20, 4758--4771  (2013).

\bibitem{De}
P. Deligne \emph{Valeurs de Fonctions L et p\'eriodes d’int\'egrales},
Proceedings of Symposia in Pure Mathematics 33, 313-346 (1979).

\bibitem{D1}
N. Diamantis,\emph{
Special values of higher derivatives of $L$-functions},
Forum Math. \textbf{11} no. 1, 229--252 (1999).

\bibitem{D}
N. Diamantis,\emph{
Hecke Operators and Derivatives of $L$-Functions},
Compositio Math. \textbf{125} no. 1, 39--54 (2001).


\bibitem{D2}
N. Diamantis,\emph{The geometry of certain cocycles associated to derivatives of $L$-functions},
Forum Math. \textbf{17} no. 5, 739--752 (2005).


\bibitem{DNS}
N. Diamantis, M. Neururer, F. Str\"omberg, \emph{A correspondence of modular forms and applications to values of $L$-series},
Res. number theory (2015) 1: 27

\bibitem{En} G. Enestr\"om, \emph{Ramarque sur un th\'eor\`eme relatif aux recines de l'equation $a_nx_n+\cdots+a_0 = 0$ o\`u tous les coefficients sont r\'eels et positifs}, T\^{o}hoku Math. J. \textbf{18}, 34--36 (1920), translation of a Swedish article in Ofversigt of Konogl. Vertenskaps Akademiens F\"orhandlingar \textbf{50}, 405--415 (1893).


\bibitem{ElG-R}
A. El-Guindy, W. Raji, \emph{Unimodularity of zeros of 
period polynomials of Hecke eigenforms}, Bull. Lond. Math. Soc. \textbf{46} no. 3, 528--536 (2014).


\bibitem{GG}
R. Gardner,  N. K. Govil, 
\emph{Enestr\"om-Kakeya theorem and some of its generalizations}, 
Current topics in pure and computational complex analysis, 171–199, 
Trends Math., Birkhäuser/Springer, New Delhi, 2014. 


\bibitem{G}
D. Goldfeld, \emph{Special values of derivatives of $L$-functions},  Number theory (Halifax, NS, 1994), 159--173, CMS Conf. Proc., \textbf{15}, Amer. Math. Soc., Providence, RI (1995).


\bibitem{GZ}
B. Gross and D. Zagier \emph{Heegner points and derivative of $L$-series}
Invent. Math. {\bf 85}, 225--320 (1986).

\bibitem{I} H. Iwaniec, {\it Topics in classical automorphic forms}
Graduate Studies in Mathematics, Vol. 17, AMS, 1991.

\bibitem{JMOS} S. Jin, W. Ma, K. Ono, and K. Soundararajan, {\it The Riemann Hypothesis for period polynomials of modular forms}, Proc. Natl. Acad. of Sci. U.S.A. \textbf{113} no. 10, 2603--2608 (2016).

\bibitem{Ka} S. Kakeya, \emph{On the limits of the roots of an algebraic equation with positive coefficients}, Tohoku Math. J. \textbf{2}, 140--142 (1912-1913).

\bibitem{KoZ} W. Kohnen and D. Zagier, {\it Modular forms with rational periods
in Modular Forms}, R.A. Rankin (ed.), Ellis Horwood, Chichester 197--249 (1984)

\bibitem{KZ} M. Kontsevich and D. Zagier, {\it Periods}, Mathematics unlimited --2001 and beyond, 771--808, Springer, Berlin, 2001. 

\bibitem{LS} M. Lal\'\i n and C. Smyth, {\it Unimodularity of zeros of self-inversive polynomials} Acta Math. Hungar. 138 (2013), no. 1-2, 85--101. Addendum, Acta Math. Hungar. {\bf 147} (2015), no. 1, 255--257.

\bibitem{M} Y. T. Manin, \emph{Periods of parabolic points and $p$-adic Hecke series,} Math. Sb.,
371--393 (1973).

\bibitem{Ma} Y. I. Manin, \emph{Local zeta factors and geometries under} Spec {\bf Z}, Izv. Russian Acad. Sci. (Volume dedicated
to J.-P. Serre) {\bf 80} no. 4, 123--130 (2016)

\bibitem{MSW} M. Murty, C. Smyth, and R. Wang, {\it Zeros of Ramanujan polynomials}
J. of the Ramanujan Math. Soc. \textbf{26}, 107--125 (2011).

\bibitem{ORS} K. Ono, L. Rolen, and F. Sprung, {\it Zeta-polynomials for modular form periods}, Adv. Math. \textbf{306}, 328--343 (2017).

\bibitem{NIST} F. Olver, D. Lozier, R. Boisvert, and
C. Clark, \textsl{NIST handbook of mathematical   
functions} U.S.
Department of Commerce, National Institute of Standards
and Technology,
Washington, DC; Cambridge University Press, Cambridge, 2010.


\bibitem{Z} D. Zagier, \emph{ Periods of modular forms and Jacobi theta functions},
Invent. Math. \textbf{104} no. 3 449--465 (1991).

\end{thebibliography}
\end{document}